\documentclass[12pt, a4paper,twosided,reqno]{amsart}
\usepackage{enumerate, cite}
\usepackage{hyperref}

\newtheorem{theorem}{Theorem}
\newtheorem{lemma}{Lemma}

\allowdisplaybreaks
\author[G. Pant and S.K.Pant]{Garima Pant and Sanjay Kumar Pant}
\address{Garima Pant; department of mathematics, university of delhi, delhi-110007, india.}
\email{garimapant.m@gmail.com}
\address{Sanjay Kumar Pant; department of mathematics, deen dayal upadhyaya college, university of delhi, new delhi-110078, india.}
\email{skpant@ddu.du.ac.in}
\thanks {Research work of the first author is supported by research fellowship from University Grants Commission (UGC), New Delhi, India.}
\title[On Solutions of Certain Non-Linear Differential-Difference Equations]{On Solutions of Certain Non-Linear Differential-Difference Equations}
\begin{document}
	\maketitle	
	
\begin{abstract}
In this paper, we study about solutions of certain kind of non-linear differential difference equations
$$f^{n}(z)+wf^{n-1}(z)f^{'}(z)+f^{(k)}(z+c)=p_{1}e^{\alpha_{1}z}+p_{2}e^{\alpha_{2}z}$$ and 
$$f^{n}(z)+wf^{n-1}(z)f^{'}(z)+q(z)e^{Q(z)}f(z+c)=p_{1}e^{\alpha_{1}z}+p_{2}e^{\alpha_{2} z},$$
where $n\geq 2$, $k\geq0$ are integers, $w, p_{1}, p_{2}, \alpha_{1}$ $\&$  $\alpha_{2}$ are non-zero constants satisfying $\alpha_{1}$ $\neq$ $\alpha_{2}$, $0\not\equiv q$ is a polynomial and $Q$ is a non-constant polynomial. 
\end{abstract}	

\section{\textbf{Introduction}}
It is assumed that readers are familiar with the standard notations of Nevanlinna theory such as $T(r,f)$, $m(r,f)$, $N(r,f)$, and $n(r,f)$ are called characteristic function of $f$, proximity function of $f$, counting function of $f$ and un-integrated counting function of $f$ respectively, here $f$ denotes a meromorphic function. Recall that for a meromorphic function $f$, Nevanlinna's first main theorem expressed as
$$T\left(r,\frac{1}{f-a}\right)=T(r,f)+O(1),$$
for all $a\in \mathbb{C}$. Here $O(1)$ denotes bounded error term depends on $a$,\cite{ilaine}. Also the quantities which are of growth $o(T(r,f))$ as $r\to \infty$, outside a set of finite linear measure, are denoted
by $S(r, f)$. It means we say that a meromorphic function $h(z)$ is said to be a small function of $f(z)$ if $T(r, h) = S(r, f)$ and converse is also true. Note that the finite sum of $S(r, f)$ quantities is again forms $S(r, f)$. \\
Next the terms order of growth $\rho(f)$, hyper-order of growth $\rho_{2}(f)$ and exponent of convergence $\lambda(f)$ of a meromorphic function $f$ are defined by
$$\rho(f)=\limsup_{r\to\infty}\frac{\log^{+} T(r,f)}{\log r},$$
$$\rho_2(f)=\limsup_{r\to\infty}\frac{\log^{+}\log^{+} T(r,f)}{\log r}$$
and
$$\lambda(f)=\limsup_{r\to\infty}\frac{\log^{+} n(r,1/f)}{\log r},$$
respectively.   In this paper, we use the above notations, definitions and facts in a frequent manner. For the standard results of Nevanlinna theory we refer \cite{ilaine,yanglo,yybook}.\\
A differential-difference polynomial $Q(z,f)$ in $f$ is  defined as a finite sum of products of $f$, derivatives of $f$ and their shifts,
with all the coefficients must be a small function of $f$, here $f$ is a meromorphic function. \\
Many researchers have been studied about the solvability and existence of solutions of certain kind of non-linear differential-difference equations, one can see \cite{chw, cl, li, li1, ly, rx, ww}. \\
In this sequence, Li and Huang studied a certain differential-difference equation
\begin{equation}\label{maineq1}
f^{n}(z)+wf^{n-1}(z)f^{'}(z)+f^{(k)}(z+c)=p_{1}e^{\alpha_{1}z}+p_{2}e^{\alpha_{2}z},
\end{equation}
where $n$ is a natural number, $k\geq 0$ is an integer, $w, p_{1}, p_{2}, \alpha_{1}$ $\&$  $\alpha_{2}$ are non-zero constant satisfying $\alpha_{1}$ $\neq$ $\alpha_{2}$, and they provided the following result:
\begin{theorem}\rm\cite{jh}
Suppose that $n\geq5$, $\alpha_{1}/\alpha_{2}\neq n$ and $\alpha_{2}/\alpha_{1}\neq n$, then equation \eqref{maineq1} has no
transcendental entire solutions.
\end{theorem}

Motivated by the above result, we prove the following result:
\begin{theorem}\label{mainth1}
Suppose that $n\geq 5$ and $f$ is a transcendental entire solution of finite order of the differential-difference equation  \eqref{maineq1},
 then it must be of $f(z)= C e^{az}$ form, where $a$ and $C$ are non-zero constants satisfying $a=\alpha_{i}$, $na=\alpha_{j}; i\neq j$ and $C=(1+aw)^{-1/n}p_{i}^{1/n}$ for some $i=1,2$.
\end{theorem}

In the next result, we add a condition $N(r,1/f)=S(r,f)$ in the hypothesis of the above theorem and we prove the same conclusion of the above theorem when $n=4$ or $3$.

\begin{theorem}\label{mainth2}
Let $n=4$ or $3$ and $f$ be a finite order transcendental entire solution of the differential-difference equation \eqref{maineq1} with $N(r,1/f)=S(r,f)$. Then $f$ must be of $f(z)= C e^{az}$ form, where $a$ and $C$ are non-zero constants satisfying $a=\alpha_{i}$, $na=\alpha_{j}; i\neq j$ and $C=(1+aw)^{-1/n}p_{i}^{1/n}$ for some $i=1,2$.
\end{theorem}

In the same paper \cite{jh}, Li and Huang studied one more type of certain differential-difference equation
\begin{equation}\label{maineq3}
f^{n}(z)+wf^{n-1}(z)f^{'}(z)+q(z)e^{Q(z)}f(z+c)=p_{1}e^{\alpha_{1}z}+p_{2}e^{\alpha_{2} z},
\end{equation}
where $n$ is a natural number, $k\geq 0$ is an integer, $w, c, p_{1}, p_{2}, \alpha_{1}, \alpha_{2}$ are non-zero constants satisfying $\alpha_{1}\neq\alpha_{2}$, $q\not\equiv 0$ is a polynomial and $Q$ is a non-constant polynomial. They provided the following result:

\begin{theorem}\rm\cite{jh}
Suppose that $n\geq 4$ and $f$ is a transcendental entire solution with finite order of equation \eqref{maineq3} with $\lambda(f)<\rho(f)$.
Then the following conclusions hold:
\begin{enumerate}[(i)]
\item Each solution $f$ satisfies $\rho(f)=\deg Q=1$.
\item If $n\geq 1$ and $f$ is a solution which belongs to $\Gamma_{0}=\{e^{\alpha(z)}:\alpha(z)$ is a non-constant polynomial\},  then
$$f(z)=e^{(\alpha_{2} z/n)+\beta}, \qquad Q(z)=(\alpha_{1}-\frac{\alpha_{2}}{n})z+b$$
or $$f(z)=e^{(\alpha_{1}z/n)+\beta}, \qquad Q(z)=(\alpha_{2}-\frac{\alpha_{1}}{n})z+b,$$
where $\beta$ and $b$ are constants.
\end{enumerate}
\end{theorem}
In the next result, we prove that the same conclusions hold when $n=3$ or $2$ under the same hypothesis as given in the above theorem.

\begin{theorem}\label{mainth3}
Suppose $n=3$ or $2$ and $f$ is a finite order transcendental entire solution  of equation \eqref{maineq3} with $\lambda(f)<\rho(f)$. Then the following conclusions hold:
\begin{enumerate}[(i)]
\item Each solution $f$ satisfies $\rho(f)=\deg Q=1$.
\item If $n\geq 1$ and $f$ is a solution which belongs to $\Gamma_{0}=\{e^{\alpha(z)}:\alpha(z)$ is a non-constant polynomial\},  then
$$f(z)=e^{(\alpha_{2} z/n)+\beta}, \qquad Q(z)=(\alpha_{1}-\frac{\alpha_{2}}{n})z+b$$
or $$f(z)=e^{(\alpha_{1}z/n)+\beta}, \qquad Q(z)=(\alpha_{2}-\frac{\alpha_{1}}{n})z+b,$$
where $\beta$ and $b$ are constants.
\end{enumerate}
\end{theorem}

Prior to \cite{jh}, Chen et.al \cite{chw} studied the equation \eqref{maineq3} when $\alpha_{2}=-\alpha_{1}$ and they didn't take  $\lambda(f)<\rho(f)$ condition in their hypothesis and proved the same conclusions.
\section{\textbf{Preliminary Results}}
The following lemma gives proximity function of logarithmic derivative of a meromorphic function $f$ . 

\begin{lemma}\rm \cite{ilaine} \label{il}
Suppose $f$ is a transcendental meromorphic function and $k$ is a natural number. Then
	$$m\left(r,\frac{f^{(k)}}{f}\right)=S(r,f).$$
If $f$ is of finite order growth, then
	$$m\left(r,\frac{f^{(k)}}{f}\right)=O(\log r).$$
\end{lemma}

Next lemma estimates the characteristic function of a shift of a meromorphic function $f$.
\begin{lemma}\rm\cite{cf}\label{cflemma}
Suppose $f$ is a meromorphic function of finite order $\rho$ and $c$ is a non-zero complex number. Then for every $\epsilon>0$,
	$$T(r, f(z+c))=T(r,f)+O(r^{\rho-1+\epsilon})+O(\log r).$$
\end{lemma}

The following lemma gives the difference analogue of the lemma on the logarithmic derivative of a meromorphic function $f$ having finite order.
\begin{lemma}\rm\cite{cf}\label{hk}
Suppose $f$ is a meromorphic function with $\rho(f)<\infty$ and $c_1,c_2\in\mathbb{C}$ such that $c_1\neq c_2$, then  for each $\epsilon>0$, we obtain
	$$m\left(r,\frac{f(z+c_1)}{f(z+c_2)}\right) =O(r^{\rho-1+\epsilon}).$$
\end{lemma}

Next lemma plays an important part in the study of complex differential-difference equations and it can be seen in \cite{yl}.
\begin{lemma}\rm\cite{yl}\label{analogueofclunielemm}
Suppose that $g$ is a transcendental entire solution of finite order of a differential-difference equation of the form
$$f^{n}P(z,g)=Q(z,g),$$
where $P(z,g)$ and $Q(z,g)$ are polynomials in $g(z)$, its derivatives and its shifts with small meromorphic coefficients. If the total degree of $Q(z,g)$ is less than or equal to $n$, then
$$m(r, P(z,g))=S(r,g),$$
for all $r$ outside of a possible exceptional set of finite logarithmic measure.
\end{lemma}

The following lemma plays a vital role in the study of uniqueness of meromorphic functions.
\begin{lemma}\label{imp}\rm\cite{yybook}
Let $f_1,f_2,...,f_n (n\geq2)$ be meromorphic functions and $h_1,h_2,...,h_n$ be entire functions satisfying 
\begin{enumerate}
\item $\sum_{i=1}^{n}f_ie^{h_i}\equiv 0$.
\item For $1\leq j<k\leq n$, $h_j-h_k$ are not constants .
\item For $1\leq i\leq n, 1\leq m<k\leq n$,\\
$T(r, f_i)=o(T(r,e^{(h_m-h_k)}))$ as $r\to\infty$, outside a set of finite logarithmic measure.
\end{enumerate}
Then $f_i\equiv 0$ $(i=1,2,...,n).$
\end{lemma}

Next lemma estimates the characteristic function of an exponential polynomial $f$. This lemma can be seen in \cite{whl}.
\begin{lemma}\label{whllem}
Suppose $f$ is an entire function given by
$$f(z)=A_{0}(z)+A_{1}(z)e^{w_{1}z^{s}}+A_{2}(z)e^{w_{2}z^{s}}+...+A_{m}(z)e^{w_{m}z^{s}},$$
where $A_{i}(z);0\leq i\leq m$ denote either exponential polynomial of degree $<s$ or polynomial in $z$, $w_{i};1\leq i\leq m$ denote the constants and $s$ denotes a natural number. Then
$$T(r,f)=C(Co(W_{0}))\frac{r^{s}}{2\pi}+o(r^{s}),$$
Here $C(Co(W_{0}))$ is the perimeter of the convex hull of the set $W_{0}=\{0,\overline{w}_{1},\overline{w}_{2},...,\overline{w}_{m}\}$.
\end{lemma}

\begin{lemma}\label{newlem}
If $f$ is a transcendental meromorphic function of finite order $\rho(f)$ and satisfying $\lambda(f)<\rho(f)$, then $N(r,1/f)=S(r,f)$.
\end{lemma}
\begin{proof}
Given that $f$ is a transcendental meromorphic function of finite order $\rho(f)$ and satisfying $\lambda(f)<\rho(f)$. We prove this lemma by contradiction.\\
 Suppose $f$ is a finite order transcendental meromorphic function such that $N(r,1/f)\neq S(r,f)$. This means $N(r,1/f)\neq o(T(r,f))$, outside a set of finite measure, and this gives
$$\lim_{r\to\infty}\frac{N(r,1/f)}{T(r,f)}\not\to 0.$$
Let 
\begin{equation}\label{newlemkeyeq}
 \lim_{r\to\infty}\frac{N(r,1/f)}{T(r,f)}\to a,
\end{equation}
where $a>0$. First we use the definition of $\limsup$ in equation \eqref{newlemkeyeq}, for every $\epsilon>0$, there exist $r_{0}$ such that  
$$\frac{N(r,1/f)}{T(r,f)}\leq a+\epsilon,$$
for $r\geq r_{0}$. This gives
\begin{equation}\label{neqlemeq1}
\lambda(f)\leq \rho(f).
\end{equation}
Similarly, we use the definition of $\liminf$ in equation\eqref{newlemkeyeq} and we get
\begin{equation}\label{neqlemeq2}
	\lambda(f)\geq \rho(f).
\end{equation}
From equation \eqref{neqlemeq1} and \eqref{neqlemeq2}, we get a contradiction to the fact $\lambda(f)<\rho(f)$.

\end{proof}

\section{\textbf{Proof of Theorems}}

\begin{proof}[\textbf{\underline{Proof of Theorem \ref{mainth1}}}]
Let $f$ be a finite order transcendental entire solution of equation $\eqref{maineq1}$.\\
Set $M=f^{n}(z)+wf^{n-1}(z)f^{'}(z)$ and $N=f^{(k)}(z+c)$, then equation $\eqref{maineq1}$ can be rewritten as
\begin{equation}\label{reducedmaineq1}
M+N=p_{1}e^{\alpha_{1}z}+p_{2}e^{\alpha_{2}z}.	
\end{equation}
Differentiating equation \eqref{reducedmaineq1}, we get
\begin{equation}\label{diffeq1}
M^{'}+N^{'}=\alpha_{1}p_{1}e^{\alpha_{1}z}+\alpha_{2}p_{2}e^{\alpha_{2}z}.
\end{equation}
From equation $\eqref{reducedmaineq1}$ and $\eqref{diffeq1}$, we eliminate $e^{\alpha_{2}z}$ and we get
\begin{equation}\label{diffeq2}
\alpha_{2}M+\alpha_{2}N-M^{'}-N^{'}=(\alpha_{2}-\alpha_{1})p_{1}e^{\alpha_{1}z}.
\end{equation}
After differentiating above equation, we get
\begin{equation}\label{diffeq3}
\alpha_{2}M^{'}+\alpha_{2}N^{'}-M^{''}-N^{''}=(\alpha_{2}-\alpha_{1})\alpha_{1}p_{1}e^{\alpha_{1}z}.
\end{equation} 
From equation $\eqref{diffeq2}$ and $\eqref{diffeq3}$, we eliminate $e^{\alpha_{2}z}$ and we get
\begin{equation}\label{diffeq4}
\alpha_{1}\alpha_{2}M-(\alpha_{1}+\alpha_{2})M^{'}+M^{''}=-(\alpha_{1}\alpha_{2}N-(\alpha_{1}+\alpha_{2})N^{'}+N^{''}),
\end{equation}
where
\begin{align*}
M&=f^{n}(z)+wf^{n-1}(z)f^{'}(z)\\
M^{'}&=nf^{n-1}f^{'}+w[(n-1)f^{n-2}(f^{'})^{2}+f^{n-1}f^{''}]\\
M^{''}&=n(n-1)f^{n-2}(f^{'})^{2}+nf^{n-1}f^{''}+(n-1)(n-2)wf^{n-3}(f^{'})^{3}+\\
& \qquad \qquad \qquad [2(n-1)+w(n-1)]f^{n-2}f^{'}f^{''}+wf^{n-1}f^{'''}.	
\end{align*}
After substituting the values of $M$, $M^{'}$ and $M^{''}$ in the equation $\eqref{diffeq4}$, we get
\begin{equation}\label{keyeq}
f^{n-3}\phi=-(\alpha_{1}\alpha_{2}N-(\alpha_{1}+\alpha_{2})N^{'}+N^{''}),
\end{equation}
where
\begin{equation}\label{impeq}
\begin{split}
\phi&=\alpha_{1}\alpha_{2}f^{3}+(w\alpha_{1}\alpha_{2}-n(\alpha_{1}+\alpha_{2}))f^{2}f^{'}+(n-w(\alpha_{1}+\alpha_{2}))f^{2}f^{''}+wf^{2}f^{'''}+\\
&\qquad (n-1)(n-w(\alpha_{1}+\alpha_{2}))f(f^{'})^{2}+w(n-1)(n-2)(f^{'})^{3}+3w(n-1)ff^{'}f^{''}.
\end{split}
\end{equation}
Given that $n-3\geq 2$ and as we set $N=f^{(k)}(z+c)$, then applying Lemma 
 \ref{analogueofclunielemm} to equation \eqref{keyeq}, we get
 \begin{equation*}
  m(r,\phi)=S(r,f) \qquad \text{and} \qquad m(r,f\phi)=S(r,f).
 \end{equation*}
Now there are two cases:\\
\textbf{1:} If $\phi\not\equiv 0$, then
\begin{align*}
T(r,f)=m(r,f)&=m\left(r,\frac{f\phi}{\phi}\right)\\
&\leq m(r,f\phi)+m\left(r,\frac{1}{\phi}\right)\\
&\leq S(r,f).	
\end{align*}
This is not possible.\\
\textbf{2:} If $\phi\equiv 0$, then from equation \eqref{impeq}, we have
\begin{equation}\label{impeq1}
\begin{split}
\alpha_{1}\alpha_{2}f^{3}&\equiv -[(w\alpha_{1}\alpha_{2}-n(\alpha_{1}+\alpha_{2}))f^{2}f^{'}+(n-w(\alpha_{1}+\alpha_{2}))f^{2}f^{''}+wf^{2}f^{'''}+\\
&(n-1)(n-w(\alpha_{1}+\alpha_{2}))f(f^{'})^{2}+w(n-1)(n-2)(f^{'})^{3}+3w(n-1)ff^{'}f^{''}].
\end{split}
\end{equation}
Suppose $f$ has infinitely many zeros, then from the above equation, it is very clear that zeros of $f$ have multiplicity greater or equal to $2$. Let $z_0$ be a zero of $f$ with multiplicity $m\geq 2$, then left side of \eqref{impeq1} has zero at $z_{0}$ with multiplicity $3m$, while right side of the same has zero at $z_{0}$ with multiplicity at most $3m-3$, which is not possible. Hence $f$ has finitely many zeros, now applying Hadamard factorisation theorem, $f$ must be of the form
\begin{equation}\label{impeqq}
f(z)=\beta(z)e^{P(z)},
\end{equation}
 where $P(z)$ is a non-constant polynomial and $\beta(z)$ is an entire function satisfying $\rho(\beta)<\deg(P)$.\\
 Now substituting the value of $f$ into equation \eqref{reducedmaineq1}, we have
 \begin{equation}\label{imp2}
 	[\beta^n(z)+w\beta^{n-1}(z)(\beta^{'}(z)+\beta(z) P^{'}(z))]e^{nP(z)}+[\gamma^{'}(z)+P^{'}(z+c)\gamma(z)]e^{P(z+c)}=p_{1}e^{\alpha_{1}z}+p_{2}e^{\alpha_{2}z},
 \end{equation} 
where $\gamma(z)$ is the coefficient of $f^{(k-1)}(z+c)$ which is in the terms of $\beta(z+c)$, $P(z+c)$ and their derivatives, and $\gamma^{'}(z)$ is the derivative of $\gamma(z)$.\\
If $\deg(P)=l\geq 2$, then applying Lemma \ref{whllem}, the order of growth of the left side of the above equation would be $l$ while right side of the same has $1$ order of growth, which is not possible. Hence $\deg P=1$, let $P(z)=az+b$, where $a$ and $b$ are constant with $a\neq 0$. After substitution the $P(z)$, equation \eqref{imp2} becomes
\begin{equation}\label{eq}
[\beta^n(z)+w\beta^{n-1}(z)(\beta^{'}(z)+a\beta(z) )]e^{n(az+b)}+[\gamma^{'}(z)+a\gamma(z)]e^{a(z+c)+b}-p_{1}e^{\alpha_{1}z}-p_{2}e^{\alpha_{2}z}=0,
\end{equation}
Next we study the following cases:
\begin{enumerate}
\item If $na\neq \alpha_{j}$ and $a\neq \alpha_{i}; i\neq j$, then applying lemma \ref{imp}, we get $p_1\equiv 0$ and $p_2\equiv 0$, which is a contradiction.
\item If $na=\alpha_{j}$ and $a\neq \alpha_{i}; i\neq j$, say $na=\alpha_1$ and $a\neq \alpha_2$, then again applying lemma \ref{imp}, we get $p_{2}\equiv 0$, which is a contradiction.
\item If $na=\alpha_{j}$ and $a= \alpha_{i}; i\neq j$, say $na=\alpha_1$ and $a= \alpha_2$, then again applying lemma \ref{imp}, we get
\begin{equation}\label{eq2}
[\beta^n(z)+w\beta^{n-1}(z)(\beta^{'}(z)+a\beta(z))]e^{nb}-p_1\equiv 0,
\end{equation}
$\implies$
$\beta^{n-1}(z)[(1+wa)\beta(z)+w\beta^{'}(z)]e^{nb}\equiv p_1$. This gives $\beta(z)$ must be a constant, say $\beta(z)=\beta_0$ and $a\neq -1/w$. Substituting this into equation \eqref{eq2}, we get
$\beta_0=(1+aw)^{-1/n}e^{-b}p^{1/n}_{1}$.  Thus from equation \eqref{impeqq}, $f(z)=(1+aw)^{-1/n}p_{1}^{1/n}e^{az}=Ce^{az}$ is the solution of equation \eqref{maineq1}.
\end{enumerate}
\end{proof}

\begin{proof}[\textbf{\underline{Proof of Theorem \ref{mainth2}}}]
Let $f$ be a finite order transcendental entire solution of equation 
\eqref{maineq1} with $N(r,1/f)=S(r,f)$.
Set $M=f^{n}(z)+wf^{n-1}(z)f^{'}(z)$ and $N=f^{(k)}(z+c)$, then equation $\eqref{maineq1}$ can be rewritten as   
\begin{equation}\label{reducedmaineq2}
	M+N=p_{1}e^{\alpha_{1}z}+p_{2}e^{\alpha_{2}z}.	
\end{equation}
Proceeding to the similar lines as we done in the proof of Theorem \ref{mainth1}, we get
\begin{equation}\label{keyeq2}
f^{n-3}\phi=-(\alpha_{1}\alpha_{2}N-(\alpha_{1}+\alpha_{2})N^{'}+N^{''}),
\end{equation}
in place of equation \eqref{keyeq},
where
\begin{equation}\label{impeq2}
\begin{split}
\phi&=\alpha_{1}\alpha_{2}f^{3}+(w\alpha_{1}\alpha_{2}-n(\alpha_{1}+\alpha_{2}))f^{2}f^{'}+(n-w(\alpha_{1}+\alpha_{2}))f^{2}f^{''}+wf^{2}f^{'''}+\\
&\qquad (n-1)(n-w(\alpha_{1}+\alpha_{2}))f(f^{'})^{2}+w(n-1)(n-2)(f^{'})^{3}+3w(n-1)ff^{'}f^{''}.
\end{split}
\end{equation}
\begin{enumerate}[(i)]
\item Let $n=4$, we study the following two cases:\\
\textbf{1:} If $\phi\not\equiv 0$, then applying Lemma 
\ref{analogueofclunielemm} to equation \eqref{keyeq2}, we get
\begin{equation}
m(r,\phi)=S(r,f).
\end{equation}
Applying Lemma \ref{il} to equation \eqref{impeq2}, we get
\begin{equation}
m\left(r,\frac{\phi}{f^{3}}\right)=S(r,f).
\end{equation}  Given that $N(r,1/f)=S(r,f)$, so
\begin{equation}
N\left( r, \frac{\phi}{f^{3}}\right)=N\left(r, \frac{1}{f^{3}}\right)=3N\left(r, \frac{1}{f}\right)=S(r,f).
\end{equation} Thus using first fundamental theorem and the above three equations, we get
\begin{align*}
T(r,f)\leq3T(r,f)&=T(r,f^{3})\\
&\leq T\left(r,\frac{f^{3}}{\phi}\right)+T(r,\phi)+O(1)\\
&= T\left(r,\frac{\phi}{f^{3}}\right)+m(r,\phi)+O(1)\\
&= S(r,f)
\end{align*}
This is not possible.\\
\textbf{2:} If $\phi\equiv 0$, then proceeding to the similar lines as we have done in the proof of Theorem \ref{mainth1}, we get the required conclusion.
\item Let $n=3$, again we study the following two cases:\\
\textbf{1:} If $\phi\not\equiv 0$, then applying Lemma 
\ref{il} and \ref{analogueofclunielemm} to equation \eqref{keyeq2}, we get
\begin{equation*}
	m\left(r,\frac{\phi}{f}\right)=S(r,f).
\end{equation*}
Applying Lemma \ref{il} to equation \eqref{impeq2}, we get
\begin{equation*}
	m\left(r,\frac{\phi}{f^{3}}\right) =S(r,f).
\end{equation*}
  Given that $N(r,1/f)=S(r,f)$, so
\begin{equation*}
N\left( r, \frac{1}{f^{3}}\right) =3N\left( r, \frac{1}{f}\right) =S(r,f).
\end{equation*} 
Thus using first fundamental theorem and the above three equations, we get
\begin{align*}
  3T(r,f)&=T(r,f^{3})\\
	&=m\left(r,\frac{1}{f^{3}}\right) +N\left( r,\frac{1}{f^{3}}\right) +O(1)\\
	&\leq m\left( r,\frac{\phi}{f^{3}}\right) +m\left( r,\frac{1}{\phi}\right) +S(r,f)\\
	&\leq T(r, \phi)+S(r,f)=m(r,\phi)+S(r,f)\\
	&\leq m\left( r,\frac{\phi}{f}\right) +m(r,f)+S(r,f)\\
	&=T(r,f)+S(r,f)
\end{align*}
This gives $2T(r,f)=S(r,f)$, which is not possible.\\
\textbf{2:} If $\phi\equiv 0$, then proceeding to the similar lines as we have done in the proof of Theorem \ref{mainth1}, we get the required conclusion.
\end{enumerate}
\end{proof}

\begin{proof}[\textbf{\underline{Proof of Theorem \ref{mainth3}}}]
Let $f$ be a finite order transcendental entire solution of equation \eqref{maineq3} satisfying $\lambda(f)<\rho(f)$. \\
Suppose $\rho(f)<1$, then applying Lemma \ref{il}, \ref{cflemma}, \ref{hk} and the first fundamental theorem of Nevanlinna to equation \eqref{maineq3}, we get
\begin{align*}
T(r, e^{Q(z)})&=T\left(r,\frac{p_{1}e^{\alpha_{1}z}+p_{2}e^{\alpha_{2}z}-f^{n}-wf^{n-1}f^{'}}{qf^{(k)}(z+c)}\right)\\
&\leq T\left(r,\frac{1}{qf^{(k)}(z+c)}\right)+T(r,p_{1}e^{\alpha_{1}z}+p_{2}e^{\alpha_{2}z})+T\left(r,f^{n}(1+\frac{f'}{f})\right)+O(1)\\
&\leq  T(r,qf^{(k)}(z+c))+T(r,p_{1}e^{\alpha_{1}z}+p_{2}e^{\alpha_{2}z})+nT(r,f)+O(\log r)\\
&\leq T(r,p_{1}e^{\alpha_{1}z}+p_{2}e^{\alpha_{2}z})+S(r,p_{1}e^{\alpha_{1}z}+p_{2}e^{\alpha_{2}z}).
\end{align*}
This gives $\deg Q(z)\leq 1$, and we know that $\deg Q(z)\geq 1$, hence $\deg Q(z)=1$ and say $Q(z)=az+b$; $a\neq 0$. Now equation \eqref{maineq3} becomes
\begin{equation*}
f^{n}(z)+wf^{n-1}(z)f^{'}(z)+q(z)e^{az+b}f(z+c)=p_{1}e^{\alpha_{1}z}+p_{2}e^{\alpha_{2} z}.
\end{equation*}
On differentiating above equation, we get
\begin{equation*}
nf^{n-1}f'+w[(n-1)f^{n-2}f'^{2}+f^{n-1}f'']+\gamma(z)e^{az+b}=p_{1}\alpha_{1}e^{\alpha_{1}z}+p_{2}\alpha_{2}e^{\alpha_{2} z},
\end{equation*}
where $\gamma(z)=q'(z)f^{(k)}(z+c)+q(z)f^{(k+1)}(z+c)+aq(z)f^{(k)}(z+c)$.\\
Eliminating $e^{\alpha_{1}z}$ from above two equations, we obtain
\begin{equation}\label{eq3.1}
	\begin{split}
\alpha_{1}f^{n}+(\alpha_{1}w-n)f^{n-1}f'-(n-1)wf^{n-2}f'^{2}-wf^{n-1}f''&+(\alpha_{1}qf^{(k)}(z+c)-\gamma(z))e^{az+b}\\&=p_{2}(\alpha_{2}-\alpha_{1})e^{\alpha_{2}z}.
\end{split}
\end{equation}
\textbf{Case 1:} If $a\neq \alpha_{2}$, then applying Lemma \ref{imp}, we get $p_{2}(\alpha_{1}-\alpha_{2})\equiv0$, which is a contradiction.\\
\textbf{Case 2:} If $a= \alpha_{2}$, then equation \eqref{eq3.1} becomes
\begin{equation}
\begin{split}
\alpha_{1}f^{n}+(\alpha_{1}w-n)f^{n-1}f'-(n-1)wf^{n-2}f'^{2}-wf^{n-1}f''+[&\{\alpha_{1}qf^{(k)}(z+c)-\gamma(z)\}e^{b}-\\
&p_{2}(\alpha_{1}-\alpha_{2})]e^{\alpha_{2}z}=0.
\end{split}
\end{equation}
Now applying Lemma \ref{imp} to the above equation, we get
$$\alpha_{1}f^{n}+(\alpha_{1}w-n)f^{n-1}f'-(n-1)wf^{n-2}f'^{2}-wf^{n-1}f''\equiv 0$$ 
or 
$$\alpha_{1}+(\alpha_{1}w-n)\frac{f'}{f}-(n-1)w(\frac{f'}{f})^{2}-w\frac{f''}{f}\equiv 0.$$
Set $\frac{f'}{f}=s$ , we get a Riccati differential equation
\begin{equation}\label{eq3.2}
(\alpha_{1}w-n)s-ws'-nws^{2}+\alpha_{1}\equiv 0, 
\end{equation}
 since $(\frac{f'}{f})'=\frac{f''}{f}-(\frac{f'}{f})^2$.\\
 By routine computation, we get $s_{1}=\frac{-1}{w}$ and $s_{2}=\frac{\alpha_{1}}{n}$ are the solutions of equation \eqref{eq3.2}.
 Let $s\neq s_{1}$ and $s\neq s_{2}$, then
 $$\frac{1}{(\frac{\alpha_{1}}{n}+\frac{1}{w})}\left(\frac{s'}{s+\frac{1}{w}}-\frac{s'}{s-\frac{\alpha_{1}}{n}}\right)=n.$$
 On integrating the above equation, we get
 $$\log \left( \frac{s+\frac{1}{w}}{s-\frac{\alpha_{1}}{n}}\right)=n\left(\frac{\alpha_{1}}{n}+\frac{1}{w}\right)+c_{1},$$
 where $c_{1}$ is an arbitrary constant.\\
 This gives
 $$s=\frac{\alpha_{1}}{n}+\frac{\alpha_{1}/n+1/w}{e^{n(\alpha_{1}/n+1/w)z+c_{1}}-1}=\frac{f'}{f}.$$
 We observe that the zeros of $e^{n(\alpha_{1}/n+1/w)z+c_{1}}-1$ are the zeros of $f$. Let $z_{0}$ be a zero of $e^{n(\alpha_{1}/n+1/w)z+c_{1}}-1$ with multiplicity $m$, then
 $$m=Res\left( \frac{f'}{f},z_{0}\right) =Res\left( \frac{\alpha_{1}}{n}+\frac{\alpha_{1}/n+1/w}{e^{n(\alpha_{1}/n+1/w)z+c_{1}}-1}, z_{0}\right) =\frac{1}{n},$$
 which is a contradiction.\\
 If $s_{1}=-1/w$, then $f'/f=-1/w$ and this gives $f=c_2e^{-z/w}$, where $c_{2}$ is an arbitrary non-zero constant. Thus $\rho(f)=1$, which is a contradiction to $\rho(f)<1$.\\
 If $s_{2}=\alpha_{1}/n$, then $f'/f=\alpha_{1}/n$ and this gives $f=c_3e^{\alpha_{1}z/n}$, where $c_{2}$ is an arbitrary non-zero constant. Thus $\rho(f)=1$, which is a contradiction.\\
 Hence $\rho(f)\geq 1$, so let $\rho(f)>1$ and set $P(z)=p_{1}e^{\alpha_{1}z}+p_{2}e^{\alpha_{2}z}$ and $H(z)=q(z)f^{(k)}(z+c)$ for simplicity. Then equation 
 \eqref{maineq3} becomes
 $$f^{n}+wf^{n-1}f'+H(z)e^{Q(z)}=P(z).$$
 On differentiating above equation gives
 $$nf^{n-1}f'+(n-1)wf^{n-2}f'^{2}+wf^{n-1}f''+(H'(z)+H(z)Q'(z))e^{Q(z)}=P'(z).$$
 Eliminating $e^{Q(z)}$ with the help of above two equation yields
 \begin{equation}\label{eq3.3}
 f^{n-2}\phi=M(z)P(z)-P'(z)H(z),
 \end{equation}
 where
 \begin{equation*}
  M(z)=H'(z)+H(z)Q'(z)
 \end{equation*}
 and 
 \begin{equation}\label{eq3.4}
 \phi(z)=M(z)f^2+wM(z)ff'-nH(z)ff'-w(n-1)H(z)f'^{2}-wH(z)ff''.
 \end{equation}
\begin{enumerate}[(i)]
\item  Let $n=3$, then applying Lemma \ref{cflemma} to equation \eqref{eq3.3}, we get
 \begin{equation}\label{eq3.5}
 m(r, \phi)=S(r,f).
 \end{equation}
 Also given that $\lambda(f)<\rho(f)$, applying Lemma \ref{newlem} gives $$N\left(r,\frac{1}{f}\right)=S(r,f),$$
  hence
  \begin{equation}\label{eq3.6}
   N\left(r, \frac{\phi}{f^3}\right)=N\left( r, \frac{1}{f^3}\right) =3N\left( r, \frac{1}{f}\right) =S(r,f).
  \end{equation}
 Applying Lemma \ref{il} and \ref{hk} to equation \eqref{eq3.4}, we get
 \begin{equation}\label{eq3.7}
 m\left( r,\frac{\phi}{f^3}\right) =S(r,f).
 \end{equation}
If $\phi\not\equiv 0$, using equations \eqref{eq3.5}, \eqref{eq3.6}, \eqref{eq3.7} and first fundamental theorem of Nevanlinna, we get
\begin{align*}
3T(r,f)=T(r,f^3)&=T\left( r, \frac{1}{f^3}\right) +O(1)\\
&\leq T\left( r, \frac{\phi}{f^3}\right) +T\left( r,\frac{1}{\phi}\right) +O(1)\\
&\leq T(r,\phi)+S(r,f) \\
&=S(r,f),
\end{align*}
which is a contradiction.\\
If $\phi\equiv 0$, from equation \eqref{eq3.3}, we get
$M(z)P(z)-P'(z)H(z)\equiv 0$.
This gives
$$(H'(z)+H(z)Q'(z))P(z)-P'(z)H(z)\equiv 0.$$
$\implies$ $$\frac{H'(z)}{H(z)}+Q'(z)-\frac{P'(z)}{P(z)}=0.$$
 $\implies$ $$\frac{q'(z)}{q(z)}+\frac{f^{(k+1)}(z+c)}{f^{(k)}(z+c)}+Q'(z)-\frac{P'(z)}{P(z)}=0.$$
 On integrating above equation, we get
 \begin{equation}\label{eq3.8}
 q(z)f^{(k)}(z+c)e^{Q(z)}=\frac{1}{c_4}P(z)=\frac{1}{c_4}(p_1e^{\alpha_{1}z}+p_1e^{\alpha_{2}z}).
 \end{equation}
Since $f$ is a finite order transcendental entire solution satisfying $\lambda(f)<\rho(f)$, then using Hadamard factorisation theorem, $f$ must be of \begin{equation}\label{eq3.9}
f(z)=g(z)e^{h(z)}
\end{equation}
form, where $h(z)$ is a polynomial such that $\rho(f)=\deg (h)>1$ and $g(z)$ is the canonical product of zeros of $f(z)$ with $\lambda(f)=\rho(g)<\rho(f)$.\\
 Using equations \eqref{eq3.8} and \eqref{eq3.9} to the equation \eqref{maineq3} gives
 
 $$g^2(z)(g'(z)+g(z)h'(z)+g(z))e^{3h(z)}=(1-1/c_{4})(p_{1}e^{\alpha_{1}z}+p_{2}e^{\alpha_{2}z}).$$
 On applying Lemma \ref{whllem} to the above equation, we get that the order of growth of the left side is greater than $1$, while the order of growth of the right side is exactly $1$. This is a  contradiction, hence $\rho(f)=1$.
 
\item Let $n=2$, then applying Lemmas \ref{il} and \ref{hk} to equations \eqref{eq3.3} and \eqref{eq3.4} give
\begin{equation}\label{4.3}
m\left(r,\frac{\phi}{f}\right)=m\left(r,\frac{H'(z)+H(z)Q'(z)}{f}\right)=S(r,f)
\end{equation}
and 
\begin{equation}\label{4.4}
	m\left( r,\frac{\phi}{f^{3}}\right) =S(r,f).
\end{equation}
Given that $\lambda(f)<\rho(f)$, applying Lemma \ref{newlem} gives
\begin{equation}\label{4.5}
	N\left( r,\frac{1}{f}\right) =S(r,f).
\end{equation}
If $\phi\not\equiv 0$, using the first fundamental theorem of Nevanlinna and equations \eqref{4.3}, \eqref{4.4} \& \eqref{4.5}, we have
\begin{align*}
	3T(r,f)=T(r,f^{3})&=m\left( r,\frac{1}{f^{3}}\right) +N\left( r,\frac{1}{f^{3}}\right) +O(1)\\
	&\leq m\left( r,\frac{\phi}{f^{3}}\right) +m\left( r,\frac{1}{\phi}\right) +3N\left( r,\frac{1}{f}\right) +O(1)\\
	&\leq T(r,\phi)+S(r,f)\\
	&= m(r,\phi)+S(r,f)\\
	&\leq m\left( r,\frac{\phi}{f}\right) +m(r,f)+S(r,f)\\
	&\leq T(r,f)+S(r,f).
\end{align*}
This gives $2T(r,f)=S(r,f)$, which is a contradiction.\\
If $\phi\equiv 0$, then proceeding similar manner as done in $(i)$, we get the same contradiction. Thus $\rho(f)=1$.\\
Next, to prove $\deg(Q)=1$ and  $(ii)$ conclusion, we follow the same technique as done in \cite[Proof of Theorem 7]{jh}.
\end{enumerate}
\end{proof}

\end{document}